\documentclass[11pt,american,english]{article}
\usepackage[T1]{fontenc}
\usepackage[latin9]{inputenc}
\usepackage{geometry}
\usepackage{color}
\usepackage{babel}
\usepackage{units}
\usepackage{amsmath}
\usepackage{amsthm}
\usepackage{amssymb}
\usepackage[all]{xy}
\usepackage[unicode=true,pdfusetitle,
 bookmarks=true,bookmarksnumbered=false,bookmarksopen=false,
 breaklinks=false,pdfborder={0 0 0},pdfborderstyle={},backref=page,colorlinks=true]
 {hyperref}
\hypersetup{
 linkcolor=blue,citecolor=blue}

\makeatletter
\theoremstyle{plain}
\newtheorem{thm}{\protect\theoremname}[section]
\theoremstyle{plain}
\newtheorem{fact}[thm]{\protect\factname}
\theoremstyle{plain}
\newtheorem{conjecture}[thm]{\protect\conjecturename}
\theoremstyle{definition}
\newtheorem{defn}[thm]{\protect\definitionname}
\theoremstyle{remark}
\newtheorem{rem}[thm]{\protect\remarkname}
\theoremstyle{plain}
\newtheorem{cor}[thm]{\protect\corollaryname}
\theoremstyle{remark}
\newtheorem{claim}[thm]{\protect\claimname}
\theoremstyle{plain}
\newtheorem{lem}[thm]{\protect\lemmaname}
\theoremstyle{plain}
\newtheorem{prop}[thm]{\protect\propositionname}

\usepackage[leftcaption]{sidecap}

\sidecaptionvpos{figure}{c}

\newcommand{\FigBesBeg}[1][1.0]{%
 \let\MyFigure\figure
 \let\MyEndfigure\endfigure
 \renewenvironment{figure}[1]{\begin{SCfigure}[#1]##1}{\end{SCfigure}}}

\newcommand{\FigBesEnd}{%
 \let\figure\MyFigure
 \let\endfigure\MyEndfigure}

\makeatother

\addto\captionsamerican{\renewcommand{\claimname}{Claim}}
\addto\captionsamerican{\renewcommand{\conjecturename}{Conjecture}}
\addto\captionsamerican{\renewcommand{\corollaryname}{Corollary}}
\addto\captionsamerican{\renewcommand{\definitionname}{Definition}}
\addto\captionsamerican{\renewcommand{\factname}{Fact}}
\addto\captionsamerican{\renewcommand{\lemmaname}{Lemma}}
\addto\captionsamerican{\renewcommand{\propositionname}{Proposition}}
\addto\captionsamerican{\renewcommand{\remarkname}{Remark}}
\addto\captionsamerican{\renewcommand{\theoremname}{Theorem}}
\addto\captionsenglish{\renewcommand{\claimname}{Claim}}
\addto\captionsenglish{\renewcommand{\conjecturename}{Conjecture}}
\addto\captionsenglish{\renewcommand{\corollaryname}{Corollary}}
\addto\captionsenglish{\renewcommand{\definitionname}{Definition}}
\addto\captionsenglish{\renewcommand{\factname}{Fact}}
\addto\captionsenglish{\renewcommand{\lemmaname}{Lemma}}
\addto\captionsenglish{\renewcommand{\propositionname}{Proposition}}
\addto\captionsenglish{\renewcommand{\remarkname}{Remark}}
\addto\captionsenglish{\renewcommand{\theoremname}{Theorem}}
\providecommand{\claimname}{Claim}
\providecommand{\conjecturename}{Conjecture}
\providecommand{\corollaryname}{Corollary}
\providecommand{\definitionname}{Definition}
\providecommand{\factname}{Fact}
\providecommand{\lemmaname}{Lemma}
\providecommand{\propositionname}{Proposition}
\providecommand{\remarkname}{Remark}
\providecommand{\theoremname}{Theorem}

\begin{document}
\global\long\def\F{\mathbb{\mathbb{\mathbf{F}}}}%
 
\global\long\def\rk{\mathbb{\mathrm{rank}}}%
 
\global\long\def\Hom{\mathrm{Hom}}%
 
\global\long\def\Epi{\mathrm{Epi}}%
 
\global\long\def\trw{{\cal T}r_{w} }%
\global\long\def\trwl{{\cal T}r_{w_{1},\ldots,w_{\ell}} }%
 
\global\long\def\tr{{\cal T}r }%
 
\global\long\def\cl{{\cal \mathrm{cl}} }%
 
\global\long\def\sq{{\cal \mathrm{sql}} }%
 
\global\long\def\auteq{\stackrel{\Aut\F}{\sim} }%
 
\global\long\def\aute#1{\stackrel{\Aut#1}{\sim} }%
 
\global\long\def\id{\mathrm{id}}%
 
\global\long\def\U{\mathcal{U}}%
 
\global\long\def\O{\mathcal{O}}%
 
\global\long\def\Aut{\mathrm{Aut}}%
 
\global\long\def\wl{w_{1},\ldots,w_{\ell}}%
 
\global\long\def\ssn{{\cal S}^{1}\wr S_{N}}%
\global\long\def\cmsn{C_{m}\wr S_{N}}%
 
\global\long\def\ctsn{C_{2}\wr S_{N}}%
 
\global\long\def\std{\mathrm{std}}%
\global\long\def\Xcov{{\scriptscriptstyle \overset{\twoheadrightarrow}{X}}}%
 
\global\long\def\covers{\leq_{\Xcov}}%
 
\global\long\def\ae{{\cal AE}}%
 
\global\long\def\alg{\le_{\mathrm{alg}}}%
 
\global\long\def\ff{\stackrel{*}{\le}}%
 
\global\long\def\ecm{\chi_{m}}%
 
\global\long\def\eci{\chi_{\infty}}%
 
\global\long\def\ect{\chi_{2}}%
 
\global\long\def\Xcov{{\scriptscriptstyle \overset{\twoheadrightarrow}{X}}}%
 
\global\long\def\covers{\leq_{\Xcov}}%
 
\global\long\def\XCO#1#2{\left[#1,#2\right)_{\Xcov}}%
 
\global\long\def\XF#1{\XCO{#1}{\infty}}%

\title{Some Orbits of Free Words that are Determined by Measures on Finite
Groups}
\author{Liam Hanany\thanks{L.~Hanany was supported by Israel Science Foundation (ISF) via grant
1071/16.}~~~~Chen Meiri\thanks{C.~Meiri was supported by BSF grant 2014099 and Israeli Science Foundation
(ISF) via grant 662/15.}~~~~Doron Puder\thanks{D.~Puder was supported by Israel Science Foundation (ISF) via grant
1071/16.}}
\maketitle
\begin{abstract}
Every word in a free group $\F$ induces a probability measure on
every finite group in a natural manner. It is an open problem whether
two words that induce the same measure on every finite group, necessarily
belong to the same orbit of $\mathrm{Aut}\F$. A special case of this
problem, when one of the words is the primitive word $x$, was settled
positively by the third author and Parzanchevski \cite{PP15}. Here
we extend this result to the case where one of the words is $x^{d}$
or $\left[x,y\right]^{d}$ for an arbitrary $d\in\mathbb{Z}$.
\end{abstract}
\tableofcontents{}

\section{Introduction\label{sec:Introduction}}

Let $r\in\mathbb{Z}_{\ge1}$, let $\F=\F_{r}$ be the free group on
$r$ generators $x_{1},\ldots,x_{r}$, and let $G$ be any finite
group. We occasionally use the letters $x$ and $y$ to denote arbitrary
distinct letters from $\left\{ x_{1},\ldots,x_{r}\right\} $. Every
word $w\in\F$ induces a map, called a word-map, 
\[
w\colon\underbrace{G\times\ldots\times G}_{r~\mathrm{times}}\to G,
\]
which is defined by substitutions. For example, if $w=x_{1}x_{3}x_{1}x_{3}^{-2}\in\F_{3}$,
then $w\left(g_{1},g_{2},g_{3}\right)=g_{1}g_{3}g_{1}g_{3}^{-2}$.
The push-forward via this word map of the uniform measure on $G\times\ldots\times G$
is called the \emph{$w$-measure} on $G$. Put differently, for each
$1\le i\le r$, substitute $x_{i}$ with an independent, uniformly-distributed
random element of $G$, and evaluate the product defined by $w$ to
obtain a random element in $G$ sampled by the $w$-measure. We say
the resulting element is a \emph{$w$-random} element of $G$.

\selectlanguage{american}%
For $w_{1},w_{2}\in\F$ write $w_{1}\aute{\F}w_{2}$ if there is an
automorphism $\theta\in\mathrm{Aut}\F$ with $\theta\left(w_{1}\right)=w_{2}$.
It is easy to see, as we explain in Section \ref{sec:profinite} below,
that applying an automorphism on $w$ does not alter the resulting
word-measure on finite groups. Thus,
\begin{fact}
\label{fact:same orbit then same measure}If $w_{1}\aute{\F}w_{2}$
then $w_{1}$ and $w_{2}$ induce the same measure on every finite
group.
\end{fact}

For example, $xyxy^{-1}\auteq x^{2}y^{2}$ and so they both induce
the same measure on every finite group. It is an open problem whether
the converse is also true, namely, whether being in the same $\Aut\F_{r}$-orbit
is the sole reason for two words to induce the same measure on every
finite group.
\begin{conjecture}
\label{conj:shalev}Let $w_{1},w_{2}\in\F_{r}$. If $w_{1}$ and $w_{2}$
induce the same measure on every finite group, then $w_{1}\aute{\F}w_{2}$.
\end{conjecture}

Conjecture \ref{conj:shalev} appears as \cite[Question 2.2]{Amit2011}
and as \cite[Conjecture 4.2]{Shalev2013}, and see also \cite[Section 8]{PP15}.
Our focus here is on special cases, where $w_{1}$ is some fixed word:
\begin{defn}
\label{def:profinitely rigid free groups}A word $w\in\F$ is called
\textbf{\emph{profinitely rigid in $\F$}} if whenever some word $w'\in\F$
induces the same measure as $w$ on every finite group, then $w\aute{\F}w'$.
The word $w\in\F$ is called \textbf{\emph{universally profinitely
rigid}} if it is profinitely rigid in every free extension $\F*\F_{\ell}$
of $\F$ ($\ell\ge1$).
\end{defn}

In Section \ref{sec:profinite} we extend this notion to arbitrary
finitely generated groups, give equivalent definitions and justify
our choice of the name ``profinitely rigid''. Note that even if
$w\in\F=\F_{r}$ is known to be profinitely rigid in $\F$, it does
not automatically follow that this property extends to words written
with more letters, namely, whether $w$ is also profinitely rigid
in $\F*\F_{\ell}$ for $\ell\ge1$. For example, it is easy to show
that the word $x$ is profinitely rigid in $\F_{1}=\F\left(x\right)\cong\mathbb{Z}$
(the orbit of $x$ in $\F_{1}$ is the sole orbit which induces measures
with full support on every finite group). However, it is much harder
to show $x$ is also profinitely rigid in $\F_{2}$ or in $\F_{r}$
for arbitrary $r$.

The fact that the trivial word $w=1$ is profinitely rigid in every
free group is equivalent to the fact that free groups are residually
finite. A case which attracted considerable attention was that of
primitive words, namely that of $\mathrm{Aut}\F.x$ -- the $\mathrm{Aut}\F$-orbit
containing the free generators of $\F$. This case was settled: first
it was shown that $x$ was profinitely rigid in $\F_{2}$ \cite{Puder2014},
and then that it was universally profinitely rigid \cite{PP15}. In
fact, it is shown in \cite{PP15} that $w\in\F$ induces the uniform
measure on the symmetric group $S_{N}$ for all $N$ if and only if
$w$ is primitive. A completely different, geometric proof of the
universal profinite rigidity of $x$ was later found by Wilton \cite[Corollary E]{wilton2018essential}.
Our main result here extends that theorem from \cite{PP15} and gives
more special cases of Conjecture \ref{conj:shalev}. Recall that $x$
and $y$ are assumed to belong to the same basis of the free group
$\F$.
\selectlanguage{english}%
\begin{thm}
\label{thm:main} For every $d\in\mathbb{Z}_{\ge1}$, the words $x^{d}$
and $\left[x,y\right]^{d}$ are universally profinitely rigid.
\end{thm}

Namely, if $w\in\F$ induces the same measure as $x^{d}$ on every
finite group, then $w\aute{\F}x^{d}$, and if it induces the same
measure as $\left[x,y\right]^{d}$, then $w\aute{\F}\left[x,y\right]^{d}$.
To the best of our knowledge, if one allows also $d=0$, Theorem \ref{thm:main}
captures all known $\Aut\F$-orbits of profinitely rigid words. We
remark that although generally $w$ and $w^{-1}$ do not necessarily
lie in the same $\mathrm{Aut}\F$-orbit, we do have $x^{-d}\auteq x^{d}$
and $\left[x,y\right]^{-d}\auteq\left[x,y\right]^{d}$, so negative
powers of $x$ and $\left[x,y\right]$ would be redundant in the statement
of Theorem \ref{thm:main}.

In a similar direction, one may consider word measures not only on
finite groups but more generally, on any compact group, where instead
of the uniform measure on $G\times\ldots\times G$ one takes the Haar
measure. Fact \ref{fact:same orbit then same measure} remains true
if the word ``finite'' is replaced by ``compact'' (e.g., \cite[Fact 2.5]{MP}),
and one can then ask a slightly weaker version of Conjecture \ref{conj:shalev}
where one assumes that $w_{1}$ and $w_{2}$ induce the same measure
on every \emph{compact} group. Indeed, in \cite{MP-surface-words},
Magee and the third author study this conjecture in the case that
$w_{1}$ is the surface word $\left[x_{1},y_{1}\right]\cdots\left[x_{g},y_{g}\right]$
or $x_{1}^{2}\cdots x_{g}^{2}$ for some $g\in\mathbb{Z}_{\ge1}$.
They show that if $w_{2}\in\F$ induces the same measure as a surface
word $w_{1}$ on every \emph{compact} group, then $w_{1}\aute{\F}w_{2}$.
Note that the word $\left[x,y\right]$ is a surface word which is
also covered by Theorem \ref{thm:main}. For this particular word,
Theorem \ref{thm:main} strengthens the result from \cite{MP-surface-words}
as it relies on measures on finite groups only.

\medskip{}

A main tool in our proof of Theorem \ref{thm:main} is a generalization
of \cite[Theorem 1.8]{PP15}, which deals with word measures on the
symmetric groups $S_{N}$. For $w\in\F=\F_{r}$, consider a $w$-random
permutation in $S_{N}$. The number of fixed points of a permutation
is equal to the trace of the corresponding permutation matrix. Thus,
we denote the expected number of fixed points of a $w$-random permutation
in $S_{N}$ by
\begin{equation}
\trw\left(N\right)\stackrel{\mathrm{def}}{=}\frac{1}{\left(N!\right)^{r}}\sum_{\sigma_{1},\ldots,\sigma_{r}\in S_{N}}\mathrm{tr}\left(w\left(\sigma_{1},\ldots,\sigma_{r}\right)\right).\label{eq:trw}
\end{equation}
A set of words $\left\{ u_{1},\ldots,u_{k}\right\} \in\F$ is called
\emph{free} if they admit no non-trivial relation.
\begin{thm}
\label{thm:free words inside a word}Let $w\in\F_{k}$ be a word which
is not contained in a proper free factor of $\F_{k}$. If $u_{1},\ldots,u_{k}\in\F$
are free words in $\F=\F_{r}$ which do not generate a free factor,
then for every large enough $N$, 
\begin{equation}
\tr_{w}\left(N\right)<\tr_{w\left(u_{1},\ldots,u_{k}\right)}\left(N\right).\label{eq:trw < trw(u)}
\end{equation}
\end{thm}

\begin{rem}
Some remarks are due:
\begin{enumerate}
\item If $w=x$ is the single-letter word, then Theorem \ref{thm:free words inside a word}
states that whenever $1\ne u\in\F_{r}$ is non-primitive, then 
\[
\tr_{x}\left(N\right)<\tr_{u}\left(N\right)
\]
for every large enough $N$. This yields that the word $x$ is profinitely
rigid, and indeed, a more quantitative version of this inequality
is the content of \cite[Theorem 1.8]{PP15}. In fact, Theorem \ref{thm:generalization of pi}
below gives a quantitative version of (\ref{eq:trw < trw(u)}) which
generalizes the quantitative version in \cite[Theorem 1.8]{PP15}.
\item The condition on $w$ in Theorem \ref{thm:free words inside a word}
is necessary but ``harmless''. To see it is necessary, consider
the primitive word $w=xy\in\F_{2}=\F\left(x,y\right)$. The words
$u_{1}=a^{n}$ and $u_{2}=b$ are free in $\F\left(a,b\right)$ and
generate the subgroup $\left\langle a^{n},b\right\rangle $ which
is not a free factor. Yet $w\left(u_{1},u_{2}\right)=a^{n}b$ is a
primitive element and therefore $\trw\left(N\right)=\tr_{w\left(u_{1},u_{2}\right)}\left(N\right)=1$
for every $N\ge1$. However, this condition can be easily ``bypassed'':
if $w$ is contained in a proper free factor of $\F_{k}$, find some
$w'$ with $w'\aute{\F_{K}}w$ which uses the smallest possible number
of letters, say $q<k$ letters, and apply the theorem with $w'\in\F_{q}$
in the stead of $w$.
\item The statement of Theorem \ref{thm:free words inside a word} is also
not true without the condition that $u_{1},\ldots,u_{k}$ be free
and not generate a free factor. Consider, for example, the case $w=x^{3}y^{2}$,
$u_{1}=a$ and $u_{2}=a^{-1}$ with $u_{1},u_{2}\in\F(a)\cong\mathbb{Z}$.
Then $w\left(u_{1},u_{2}\right)=a$ and for every $N\ge3$,
\[
1+\frac{1}{N-1}=\tr_{w}\left(N\right)>\tr_{w\left(u_{1},u_{2}\right)}\left(N\right)=1.
\]
Moreover, with the same $w$, if we take $u_{1}=w$ and $u_{2}=w^{-1}$,
we get $w\left(u_{1},u_{2}\right)=w$. Finally, if $u_{1},\ldots,u_{k}$
are free but generate a free factor of $\F$ then $w\aute{\F}w\left(u_{1},\ldots,u_{k}\right)$
whence $\tr_{w}\left(N\right)=\tr_{w\left(u_{1},\ldots,u_{k}\right)}\left(N\right)$
for all $N$.
\end{enumerate}
\end{rem}

Another ingredient of the proof of Theorem \ref{thm:main} concerns
powers of general words. We prove, in fact, that profinite rigidity
is preserved under taking powers:
\begin{thm}
\label{thm:from word to its powers}Let $w\in\F$ and $d\in\mathbb{Z}$.
If $w$ is profinitely rigid then so is $w^{d}$.
\end{thm}

In particular, if $w$ is profinitely rigid, so is $w^{-1}$, although
this case is immediate from the definition. (Note that generally,
$w$ and $w^{-1}$ do not belong to the same $\mathrm{Aut}\F$-orbit.)
We stress that the proof of Theorem \ref{thm:from word to its powers}
uses only the theory of profinite groups and free products, and does
not rely on measures induced on groups.

As explained in Section \ref{sec:profinite} below, profinite rigidity
of words yields two additional properties concerning the profinite
completion of the free group and its profinite topology. We state
these properties in the following corollary:
\begin{cor}
\label{cor:profinite}Let $w=x^{d}$ or $w=\left[x,y\right]^{d}$
for some $d\in\mathbb{Z}_{\ge1}$. Then,
\begin{enumerate}
\item \label{enu:closed in the prof top}The $\Aut\F$-orbit of $w$ is
closed in the profinite topology of $\F$.
\item \label{enu:aut-F-hat equivalence}Let $\widehat{\F}$ be the profinite
completion of $\F$. If $w'\in\F$ is in the same $\Aut\widehat{\F}$-orbit
as $w$, then $w'\aute{\F}w$.
\end{enumerate}
\end{cor}

In fact, Theorem \ref{thm:six equivalences} below shows that profinite
rigidity of $w$ is equivalent to the statement of Corollary \ref{cor:profinite}(\ref{enu:aut-F-hat equivalence}).
Item \ref{enu:profinitely rigid} is a special case of Claim \ref{claim:the orbit of prof rigid elem is closed in prof top}
below.

\subsection*{Reducing Theorem \ref{thm:main} to Theorems \ref{thm:free words inside a word}
and \ref{thm:from word to its powers}}

Theorem \ref{thm:from word to its powers} reduces Theorem \ref{thm:main}
to the primitive case, which is shown in \cite{PP15}, and to the
commutator word $\left[x,y\right]$. For the latter, we use a result
of Khelif:
\begin{thm}
\label{thm:Khelif}\cite{khelif2004finite} If the image of $w\in\F$
in every finite quotient of $\F$ is a commutator, then $w$ is a
commutator, namely, $w=\left[u,v\right]$ for some $u,v\in\F$.
\end{thm}

Now assume that some word $w'\in\F$ induces the same measures on
finite groups as the commutator $\left[u,v\right]$ for some $u,v\in\F$.
If $w'$ is not a commutator in some finite quotient $Q$ of $\F_{r}$,
then its image $\overline{w'}\in Q$ is in the support of the $w'$-measure
but not in the support of the $\left[u,v\right]$-measure, in contradiction.
Hence $w'$ is a commutator in every finite quotient of $\F_{r}$,
and by Khelif, $w'=\left[u',v'\right]$ for some $u',v'\in\F$. We
obtain the following:
\begin{cor}
\label{cor:from-Khelif} If $w_{1}\in\F$ is a commutator and $w_{2}\in\F$
is not, then $w_{1}$ and $w_{2}$ do \emph{not} induce the same measure
on all finite groups.
\end{cor}

\label{using thm 1.5 for =00005Bx,y=00005D}Now, assume $w\in\F$
induces the same measures on finite groups as the word $\left[x,y\right]$.
By Khelif's result, $w$ is a commutator, so $w=\left[u,v\right]$
for some $u,v\in\F$. Clearly, $w\ne1$, whence $u$ and $v$ do not
commute and are therefore free. From Theorem \ref{thm:free words inside a word}
applied with $\left[x,y\right]$, it immediately follows that $\left\langle u,v\right\rangle $
is a free factor of rank $2$ of $\F$ and therefore $\left(u,v\right)\auteq\left(x,y\right)$
and $w\auteq\left[x,y\right]$.

We remark that the case of primitive powers follows also from the
combination of Theorem \ref{thm:free words inside a word} and Lubotzky's
Theorem \ref{thm:nth powers are closed} (see below), which yields,
analogously to Corollary \ref{cor:from-Khelif}, that if $w_{1}$
is a $d$th power and $w_{2}$ is not, then $w_{1}$ and $w_{2}$
do not induce the same measure on finite groups. Indeed, if $w$ induces
the same measures on finite groups as $x^{d}$, then by Lubotzky's
theorem, $w=u^{d}$ with $u\ne1$ a non-power. But $\tr_{u^{d}}\left(N\right)=\tr_{x^{d}}\left(N\right)$
for all $N$, so by Theorem \ref{thm:free words inside a word}, $u$
must be primitive. See also Proposition \ref{prop:primitive powers from S_N}
below.

\subsection*{Paper organization}

Section \ref{sec:profinite} contains a short introduction to profinite
topology and to profinite groups, generalizes the notion of profinitely
rigid elements to arbitrary finitely generated groups, and gives several
equivalent notions in Theorem \ref{thm:six equivalences}. In Section
\ref{sec:Fixed-points-of} we prove Theorem \ref{thm:free words inside a word}
about the average number of fixed points in a $w\left(u_{1},\ldots,u_{k}\right)$-random
permutation. Finally, Section \ref{sec:powers} contains the proof
of Theorem \ref{thm:from word to its powers}, thus concluding the
proof of our main theorem, Theorem \ref{thm:main}.

\section{Profinitely rigid elements and equivalent notions\label{sec:profinite}}

Given a basis $x_{1},\ldots,x_{r}$ to $\F$ as above, there is a
natural correspondence 
\[
\Hom\left(\F,G\right)~\longleftrightarrow~\underbrace{G\times\ldots\times G}_{r~\mathrm{times}},
\]
where $\varphi\in\Hom\left(\F,G\right)$ corresponds to the $r$-tuple
$\left(\varphi\left(x_{1}\right),\ldots,\varphi\left(x_{r}\right)\right)$.
In this language, the $w$-measure on $G$ is the distribution of
$\varphi\left(w\right)$ where $\varphi\in\Hom\left(\F,G\right)$
is a uniformly random homomorphism. Assume that $w_{1},w_{2}\in\F$
are in the same $\Aut\F$-orbit, namely, that there exists $\theta\in\Aut\F$
with $\theta\left(w_{1}\right)=w_{2}$. If $\varphi\in\Hom\left(\F,G\right)$
is uniformly random, then so is $\varphi\circ\theta\in\Hom\left(\F,G\right)$.
Clearly, for every fixed homomorphism $\varphi$, we have $\varphi\left(w_{2}\right)=\left(\varphi\circ\theta\right)\left(w_{1}\right)$,
which proves Fact \ref{fact:same orbit then same measure}: $w_{1}$
and $w_{2}$ induce the same measure on every finite group.

In fact, this last observation is relevant not only to finitely generated
free groups, but to arbitrary finitely generated groups. Let $\Gamma$
be a finitely generated group and $G$ some finite group. The set
of homomorphisms $\Hom\left(\Gamma,G\right)$ is finite, and so every
element $\gamma\in\Gamma$ induces a measure on $G$ defined by the
random element $\varphi\left(\gamma\right)$ where $\varphi\colon\Gamma\to G$
is a uniformly random homomorphism. The previous paragraph yields
the following generalization of Fact \ref{fact:same orbit then same measure}:
\begin{claim}
\label{claim:same orbit =00003D=00003D> same measures}Let $\Gamma$
be a finitely generated group and $\gamma_{1},\gamma_{2}\in\Gamma$.
If $\gamma_{1}\aute{\Gamma}\gamma_{2}$ then $\gamma_{1}$ and $\gamma_{2}$
induce the same measure on every finite group.
\end{claim}

The study of measures induced on finite groups by words, and more
generally by elements of a finitely generated group $\Gamma$, is
closely related to some aspects of the profinite topology on $\Gamma$
and of its profinite completion. The standard references to the theory
of profinite groups are the books \cite{wilson1998profinite,ribes2010profinite}.
Let us give here some basic definitions and facts.

The profinite topology on $\Gamma$ is defined by the basis of (say,
left) cosets of subgroups of finite index. The profinite completion
of $\Gamma$, denoted $\hat{\Gamma}$\marginpar{$\hat{\Gamma}$},
is the inverse limit 
\begin{equation}
\varprojlim_{N\trianglelefteq_{\mathrm{f.i.}}\Gamma}\nicefrac{\Gamma}{N},\label{eq:inverse limit}
\end{equation}
where $N$ runs over all normal subgroups of finite index in $\Gamma$.
This is a (Hausdorff, compact, totally disconnected) topological group.
There is a natural homomorphism $\iota\colon\Gamma\to\hat{\Gamma}$
defined by mapping $\gamma\in\Gamma$ to the element $\left(\gamma N\right)_{N\trianglelefteq{}_{\mathrm{f.i.}}\Gamma}$
in the inverse limit (\ref{eq:inverse limit}). The homomorphism $\iota$
is injective if and only if $\Gamma$ is residually finite. The image
of $\Gamma$ is dense in $\hat{\Gamma}$ \cite[Lemma 3.2.1]{ribes2010profinite}.

By definition, a homomorphism from $\hat{\Gamma}$ to a finite group
is assumed to be continuous, and the sets $\Hom\left(\hat{\Gamma},G\right)$
and $\Epi\left(\hat{\Gamma},G\right)$ are the sets of \emph{continuous
}homomorphisms and epimorphisms, respectively, from $\hat{\Gamma}$
to the finite group $G$. For every finite group $G$, there is a
one-to-one correspondence between $\Hom\left(\hat{\Gamma},G\right)$
and $\Hom\left(\Gamma,G\right)$, given by $\psi\mapsto\psi\circ\iota$:
this is due to the universal property of the profinite completion
of a group, namely, for every $\varphi\in\Hom\left(\Gamma,G\right)$
there is a unique $\hat{\varphi}\in\Hom\left(\hat{\Gamma},G\right)$
such that $\varphi=\hat{\varphi}\circ\iota$ (see \cite[Lemma 3.2.1]{ribes2010profinite}).
Similarly, there is a one-to-one correspondence between $\Epi\left(\hat{\Gamma},G\right)$
and $\Epi\left(\Gamma,G\right)$.

Let $\gamma\in\Gamma$ and $g\in G$ where $\Gamma$ is finitely generated
and $G$ finite. Define
\begin{eqnarray}
\Hom_{\gamma,g}\left(\Gamma,G\right) & \stackrel{\mathrm{def}}{=} & \left\{ \varphi\in\Hom\left(\Gamma,G\right)\,\middle|\,\varphi\left(\gamma\right)=g\right\} ,\nonumber \\
\mathrm{Epi}_{\gamma,g}\left(\Gamma,G\right) & \stackrel{\mathrm{def}}{=} & \left\{ \varphi\in\mathrm{Epi}\left(\Gamma,G\right)\,\middle|\,\varphi\left(\gamma\right)=g\right\} ,\nonumber \\
\mathrm{EpiIm}_{\gamma}\left(\Gamma,G\right) & \stackrel{\mathrm{def}}{=} & \left\{ \varphi\left(\gamma\right)\,\middle|\,\varphi\in\Epi\left(\Gamma,G\right)\right\} ,~\mathrm{and}\nonumber \\
K_{\Gamma}\left(G\right) & \stackrel{\mathrm{def}}{=} & \bigcap_{N\trianglelefteq\Gamma,\nicefrac{\Gamma}{N}\cong G}N,\label{eq:def of K_Gamma_G}
\end{eqnarray}
(if $G$ is not a quotient of $\Gamma$, define $K_{\Gamma}\left(G\right)\stackrel{\mathrm{def}}{=}\Gamma$).
Note that $K_{\Gamma}\left(G\right)$ is a characteristic finite index
subgroup of $\Gamma$, as there are finitely many normal subgroups
$N\trianglelefteq\Gamma$ with quotient $G$, and as every automorphism
of $\Gamma$ permutes these normal subgroups.

The following theorem gives a number of equivalent definitions to
a relation between different elements of $\Gamma$. The automorphism
group $\Aut\hat{\Gamma}$ is the group of all \emph{continuous }automorphisms
of $\hat{\Gamma}$.
\begin{thm}
\label{thm:six equivalences}Let $\Gamma$ be a finitely generated
group, and $\gamma_{1},\gamma_{2}\in\Gamma$. Then the following six
properties are equivalent:
\begin{enumerate}
\item \label{enu:profinitely rigid}$\gamma_{1}\aute{\hat{\Gamma}}\gamma_{2}$,
namely, there exists an automorphism $\theta\in\Aut\hat{\Gamma}$
with $\theta\left(\iota\left(\gamma_{1}\right)\right)=\iota\left(\gamma_{2}\right)$.
\item \label{enu:same measure on finite groups}$\left|\Hom{}_{\gamma_{1},g}\left(\Gamma,G\right)\right|=\left|\Hom_{\gamma_{2},g}\left(\Gamma,G\right)\right|$
for every finite group $G$ and every $g\in G$, namely, $\gamma_{1}$
and $\gamma_{2}$ induce the same measure on every finite group.
\item \label{enu:same measures of epimorphisms} $\left|\mathrm{Epi}{}_{\gamma_{1},g}\left(\Gamma,G\right)\right|=\left|\mathrm{Epi}_{\gamma_{2},g}\left(\Gamma,G\right)\right|$
for every finite group $G$ and every $g\in G$.
\item \label{enu:same image in epimorphisms} $\mathrm{Epi\mathrm{Im}}{}_{\gamma_{1}}\left(\Gamma,G\right)=\mathrm{Epi}\mathrm{Im}_{\gamma_{2}}\left(\Gamma,G\right)$
for every finite group $G$, namely, $\gamma_{1}$ and $\gamma_{2}$
have the same possible images under epimorphisms to finite groups.
\item \label{enu:equiv in every special characteristic quotient} $\gamma_{1}K\aute{\left(\nicefrac{\Gamma}{K}\right)}\gamma_{2}K$
with $K=K_{\Gamma}\left(G\right)$ for every finite group $G$.
\item \label{enu:equiv in every large enough quotient}For every $N\triangleleft_{\mathrm{f.i.}}\Gamma$
there exists $K\trianglelefteq_{\mathrm{f.i.}}\Gamma$ with $K\le N$
such that $\gamma_{1}K\aute{\left(\nicefrac{\Gamma}{K}\right)}\gamma_{2}K$.
\end{enumerate}
\end{thm}

This theorem explains why the following definition generalizes Definition
\ref{def:profinitely rigid free groups}.
\begin{defn}
\label{def:automorphicly rigid - general}Let $\Gamma$ be a finitely
generated group. An element $\gamma\in\Gamma$ is called \textbf{profinitely
rigid }if whenever $\gamma'\in\Gamma$ satisfies $\gamma\aute{\hat{\Gamma}}\gamma'$
then $\gamma\aute{\Gamma}\gamma'$.
\end{defn}

Note that by Claim \ref{claim:same orbit =00003D=00003D> same measures},
the six equivalences in Theorem \ref{thm:six equivalences}, including
$\gamma_{1}\aute{\hat{\Gamma}}\gamma_{2}$, are all a consequence
of $\gamma_{1}\aute{\Gamma}\gamma_{2}$. It is natural to generalize
Conjecture \ref{conj:shalev} and ask which finitely generated groups
have the property that each of their elements is profinitely rigid.
This property holds trivially for finite groups and easily for finitely
generated abelian groups (see, for instance, \cite[Theorem 5.2]{Collins2020automorphism-invariant},
for the case of $\mathbb{Z}^{r}$). As mentioned above, this property
is (much) stronger then residually finiteness. 

The proof of Theorem \ref{thm:six equivalences} relies on the following
lemma, which is a cousin of \cite[Proposition 4.4.3]{ribes2010profinite}:
\begin{lem}
\label{lem:Aut(hat(G)) as inverse limit}Let $\Gamma$ be finitely
generated. Then 
\begin{equation}
\Aut\hat{\Gamma}\cong\varprojlim_{K}\Aut\left(\nicefrac{\Gamma}{K}\right),\label{eq:iso of Aut(profinite)}
\end{equation}
where the inverse limit is taken over all subgroups $K\le\Gamma$
such that $K=K_{\Gamma}\left(G\right)$ (defined in (\ref{eq:def of K_Gamma_G})
above) for some finite group $G$, with arrows $K_{1}\to K_{2}$ whenever
$K_{1}\le K_{2}$.
\end{lem}

\begin{proof}
Assume that $K_{\Gamma}\left(G_{1}\right)=K_{1}\le K_{2}=K_{\Gamma}\left(G_{2}\right)$,
and denote $Q_{i}=\nicefrac{\Gamma}{K_{i}}$ for $i=1,2$. The image
of $K_{2}$ in $Q_{1}$ is equal to $K_{Q_{1}}\left(G_{2}\right)$,
whence this image is characteristic in $Q_{1}$. Therefore, there
is a well-defined homomorphism $\Aut\left(Q_{1}\right)\to\Aut\left(Q_{2}\right)$.
In addition, for every $K_{1}=K_{\Gamma}\left(G_{1}\right)$ and $K_{2}=K_{\Gamma}\left(G_{2}\right)$,
define $G_{3}=\nicefrac{\Gamma}{K_{1}\cap K_{2}}$ and then $K_{3}=K_{\Gamma}\left(G_{3}\right)\le K_{1},K_{2}$.
Therefore the right hand side of (\ref{eq:iso of Aut(profinite)})
is a well-defined inverse system.

By \cite[Proposition 3.2.2]{ribes2010profinite}, for $K\trianglelefteq_{\mathrm{f.i.}}\Gamma$
we have $\nicefrac{\Gamma}{K}\cong\nicefrac{\hat{\Gamma}}{\overline{\iota\left(K\right)}}$,
and $\iota\left(K_{\Gamma}\left(G\right)\right)=K_{\hat{\Gamma}}\left(G\right)$.
It is therefore enough to show that
\begin{equation}
\Aut\hat{\Gamma}\cong\varprojlim_{K}\Aut\left(\nicefrac{\hat{\Gamma}}{K}\right),\label{eq:iso of Aut(profinite)-1}
\end{equation}
where the inverse limit runs over all subgroups $K\le\hat{\Gamma}$
such that $K=K_{\hat{\Gamma}}\left(G\right)$ for some finite group
$G$. As $K=K_{\hat{\Gamma}}\left(G\right)$ is characteristic in
$\hat{\Gamma}$, every automorphism of $\hat{\Gamma}$ induces an
automorphism of $\nicefrac{\hat{\Gamma}}{K}$ which agrees with the
inverse system, so there is a natural continuous homomorphism 
\[
\omega\colon\Aut\hat{\Gamma}\to\varprojlim_{K}\Aut\left(\nicefrac{\hat{\Gamma}}{K}\right).
\]
The map $\omega$ is injective because $\bigcap_{G~\mathrm{finite}}K_{\hat{\Gamma}}\left(G\right)=\left\{ e_{\hat{\Gamma}}\right\} $.
The map $\omega$ is surjective because every element of the inverse
system $\varprojlim_{K}\Aut\left(\nicefrac{\hat{\Gamma}}{K}\right)$
defines a continuous automorphism of $\varprojlim_{K}\nicefrac{\hat{\Gamma}}{K}\cong\hat{\Gamma}$.
\end{proof}

\begin{proof}[Proof of Theorem \ref{thm:six equivalences}]
~\\
\textbf{The implication \ref{enu:profinitely rigid}$\Longrightarrow$\ref{enu:same measure on finite groups}:}
Notice that by the one-to-one correspondence mentioned above between\linebreak{}
$\Hom\left(\Gamma,G\right)$ and $\Hom\left(\hat{\Gamma},G\right)$,
we have $\left|\Hom_{\gamma,g}\left(\Gamma,G\right)\right|=\left|\Hom_{\iota\left(\gamma\right),g}\left(\hat{\Gamma},G\right)\right|$.
Hence

\[
\left|\Hom_{\gamma_{1},g}\left(\Gamma,G\right)\right|=\left|\Hom_{\iota\left(\gamma_{1}\right),g}\left(\hat{\Gamma},G\right)\right|=\left|\Hom_{\iota\left(\gamma_{2}\right),g}\left(\hat{\Gamma},G\right)\right|=\left|\Hom_{\gamma_{2},g}\left(\Gamma,G\right)\right|,
\]
where the proof of the middle equality is identical to the proof of
Claim \ref{claim:same orbit =00003D=00003D> same measures}.

\noindent \textbf{The equivalence \ref{enu:same measure on finite groups}$\Longleftrightarrow$\ref{enu:same measures of epimorphisms}}
goes by induction on the cardinality of the finite group $G$, as
the case $\left|G\right|=1$ is trivial, and
\[
\Hom_{\gamma,g}\left(\Gamma,G\right)=\bigsqcup_{H\le G\colon g\in H}\Epi_{\gamma,g}\left(\Gamma,H\right).
\]

\noindent \textbf{The implication \ref{enu:same measures of epimorphisms}$\Longrightarrow$\ref{enu:same image in epimorphisms}}
is evident: $g\in\mathrm{EpiIm}_{\gamma}\left(\Gamma,G\right)$ if
and only if $\left|\Epi_{\gamma,g}\left(\Gamma,G\right)\right|>0$.

\noindent \textbf{The implication \ref{enu:same image in epimorphisms}$\Longrightarrow$\ref{enu:equiv in every special characteristic quotient}:}
Let $K=K_{\Gamma}\left(G\right)$ for some finite group $G$. As $\Epi\left(\Gamma,\nicefrac{\Gamma}{K}\right)\ne\emptyset$,
also\linebreak{}
$\Epi\mathrm{Im}_{\gamma_{1}}\left(\Gamma,\nicefrac{\Gamma}{K}\right)=\Epi\mathrm{Im}_{\gamma_{2}}\left(\Gamma,\nicefrac{\Gamma}{K}\right)\ne\emptyset$.
Choose an arbitrary $q\in\mathrm{EpiIm}_{\gamma_{1}}\left(\Gamma,\nicefrac{\Gamma}{K}\right)$,
and for $i=1,2$ let $f_{i}\in\Epi_{\gamma_{i},q}\left(\Gamma,\nicefrac{\Gamma}{K}\right)$.
Clearly, $\nicefrac{\Gamma}{\ker f_{i}}\cong\nicefrac{\Gamma}{K}$,
but we claim that $\ker f_{i}=K$. Indeed, the number of normal subgroups
in $\nicefrac{\Gamma}{\ker f_{i}}$ with quotient $G$ is the same
as in $\nicefrac{\Gamma}{K}$ and thus the same as in $\Gamma$. Hence
$\ker f_{i}$ is contained in every $N\trianglelefteq\Gamma$ with
$\nicefrac{\Gamma}{N}\cong G$, and so $\ker f_{i}\le K$, but $\left[\Gamma:\ker f_{i}\right]=\left[\Gamma:K\right]$,
whence $\ker f_{i}=K$. We deduce that $f_{i}$ induces an automorphism
$\overline{f_{i}}\in\mathrm{Aut}\left(\nicefrac{\Gamma}{K}\right)$,
and $\overline{f_{2}}^{-1}\circ\overline{f_{1}}$ is an automorphism
mapping $\gamma_{1}K$ to $\gamma_{2}K$.

\noindent \textbf{The equivalence \ref{enu:equiv in every special characteristic quotient}$\Longleftrightarrow$\ref{enu:equiv in every large enough quotient}:
}Assume first that $\gamma_{1}K\aute{\left(\nicefrac{\Gamma}{K}\right)}\gamma_{2}K$
for every $K=K_{\Gamma}\left(G\right)$. If $N\trianglelefteq_{\mathrm{f.i.}}\Gamma$,
then $K=K_{\Gamma}\left(\nicefrac{\Gamma}{N}\right)$ would work for
\ref{enu:equiv in every large enough quotient}. Conversely, assume
item \ref{enu:equiv in every large enough quotient} holds. Let $K=K_{\Gamma}\left(G\right)$
for some finite group $G$. By assumption, there exists a subgroup
$K'\le K$ such that $K'\trianglelefteq_{\mathrm{f.i.}}\Gamma$ and
$\gamma_{1}K'\aute{\left(\nicefrac{\Gamma}{K'}\right)}\gamma_{2}K'$.
But the image of $K$ in $Q'=\text{\ensuremath{\nicefrac{\Gamma}{K'}}}$
is precisely $K_{Q'}\left(G\right)$, so this image is characteristic,
and so every automorphism of $Q'=\nicefrac{\Gamma}{K'}$ induces an
automorphism of $\nicefrac{\Gamma}{K}$. Hence $\gamma_{1}K\aute{\left(\nicefrac{\Gamma}{K}\right)}\gamma_{2}K$.

\noindent \textbf{The implication \ref{enu:equiv in every special characteristic quotient}$\Longrightarrow$\ref{enu:profinitely rigid}:
}By assumption, for every $K$ in the inverse system in (\ref{eq:iso of Aut(profinite)}),
the subset\linebreak{}
$\mathrm{\Aut_{\gamma_{1},\gamma_{2}}\left(\nicefrac{\Gamma}{K}\right)\subseteq\Aut\left(\nicefrac{\Gamma}{K}\right)}$
of automorphisms mapping $\gamma_{1}$ to $\gamma_{2}$ is not empty.
By a standard compactness argument, there is an element in $\varprojlim_{K}\Aut\left(\nicefrac{\Gamma}{K}\right)$
mapping $\gamma_{1}K$ to $\gamma_{2}K$ for every $K=K_{\Gamma}\left(G\right)$.
We are done by Lemma \ref{lem:Aut(hat(G)) as inverse limit}.
\end{proof}
The following property of profinitely rigid elements is a generalization
of Corollary \ref{cor:profinite}(\ref{enu:closed in the prof top}):
\begin{claim}
\label{claim:the orbit of prof rigid elem is closed in prof top}Let
$\Gamma$ be a finitely generated group and $\gamma\in\Gamma$ a profinitely
rigid element. Then the orbit $\mathrm{Aut}\Gamma$$\left(\gamma\right)$
is closed in the profinite topology on $\Gamma$.
\end{claim}

\begin{proof}
We give two proofs for this claim. First, the automorphism group $\mathrm{Aut}\hat{\Gamma}$
is a profinite group (\cite[Proposition 4.4.3]{ribes2010profinite}
or Lemma \ref{lem:Aut(hat(G)) as inverse limit} above), and in particular
compact, so $\mathrm{Aut}\hat{\Gamma}\left(\gamma\right)$ is closed
in $\hat{\Gamma}$. As $\gamma$ is profinitely rigid, $\mathrm{Aut}\Gamma\left(\gamma\right)=\mathrm{Aut}\hat{\Gamma}\left(\gamma\right)\cap\Gamma$,
which shows that $\mathrm{Aut}\Gamma\left(\gamma\right)$ is closed
in the profinite topology on $\Gamma$.

The second proof uses item \ref{enu:equiv in every special characteristic quotient}
in Theorem \ref{thm:six equivalences}: if $\gamma$ is profinitely
rigid, then for every $\delta\in\Gamma\setminus\mathrm{Aut}\Gamma\left(\gamma\right)$,
there is some finite group $G$ so that $\gamma K\stackrel{\mathrm{Aut}\left(\nicefrac{\Gamma}{K}\right)}{\not\sim}\delta K$
with $K=K_{\Gamma}\left(G\right)$. Then $\delta K$ in an open neighborhood
of $\delta$ in $\Gamma$ which is disjoint from $\mathrm{Aut}\Gamma\left(\gamma\right)$.
\end{proof}
\begin{rem}
We remark on several relations between words in free groups that lie
between $\aute{\F}$ and $\aute{\hat{\F}}$. The implications between
them are described in the following diagram.
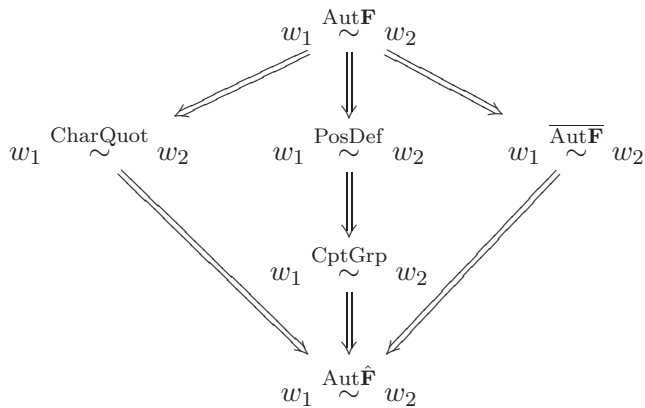
\[
\xymatrix{ & w_{1}\auteq w_{2}\ar@{=>}[ld]\ar@{=>}[d]\ar@{=>}[rd]\\
w_{1}\stackrel{\mathrm{CharQuot}}{\sim}w_{2}\ar@{=>}[rdd] & w_{1}\stackrel{\mathrm{PosDef}}{\sim}w_{2}\ar@{=>}[d] & w_{1}\stackrel{\overline{\mathrm{Aut}\F}}{\sim}w_{2}\ar@{=>}[ldd]\\
 & w_{1}\stackrel{\mathrm{CptGrp}}{\sim}w_{2}\ar@{=>}[d]\\
 & w_{1}\aute{\hat{\F}}w_{2}
}
\]
Here, $w_{1}\stackrel{\mathrm{CharQuot}}{\sim}w_{2}$ means $w_{1}K\aute{\left(\nicefrac{\F}{K}\right)}w_{2}K$
for every characteristic subgroup $K\mathrm{char}\F$. The relation
$w_{1}\stackrel{\mathrm{PosDef}}{\sim}w_{2}$, introduced in \cite{Collins2020automorphism-invariant},
means that $\tau\left(w_{1}\right)=\tau\left(w_{2}\right)$ for all
$\mathrm{Aut}\F$-invariant positive definite functions $\tau$ on
$\F$. We write $w_{1}\stackrel{\mathrm{CptGrp}}{\sim}w_{2}$ to mean
that that $w_{1}$ and $w_{2}$ induce the same measure on every compact
group, and write $w_{1}\stackrel{\overline{\mathrm{Aut}\F}}{\sim}w_{2}$
to mean that there is an automorphism $\theta$ of $\hat{\F}$ which
lies in the closure of $\mathrm{Aut}\F$ in $\mathrm{Aut}\hat{\F}$
so that $\theta\left(w_{1}\right)=\theta\left(w_{2}\right)$. The
one implication that is not immediate, $w_{1}\stackrel{\mathrm{PosDef}}{\sim}w_{2}\Longrightarrow w_{1}\stackrel{\mathrm{CptGrp}}{\sim}w_{2}$
is explained in \cite[Lemma 1.12]{Collins2020automorphism-invariant}.
Of course, Conjecture \ref{conj:shalev} implies that all these relations
are equivalent.

All these relations, except for $w_{1}\stackrel{\mathrm{CptGrp}}{\sim}w_{2}$,
can be immediately generalized for elements of every finitely generated
group.
\end{rem}

\section{\label{sec:Fixed-points-of}Fixed points of random permutations and
the proof of Theorem \ref{thm:free words inside a word}}

The proof of Theorem \ref{thm:free words inside a word} relies on
the partial order defined by ``algebraic extensions'' on the set
of (finitely generated) subgroups of the free group $\F$. We begin
with a short presentation of this notion.
\selectlanguage{american}%

\subsection{Algebraic extensions}

\selectlanguage{english}%
Let $\F$ be a free group as above and $H,J\le\F$ two subgroups.
We call $J$ an algebraic extension of $H$, denoted $H\alg J$\marginpar{$H\protect\alg J$},
if and only if $H\le J$ and there is no intermediate proper free
factor of $J$, namely, if whenever $H\le M\ff J$, we must have $M=J$
(here $M\ff J$ means that $M$ is a free factor of $J$). We collect
some of the properties of this notion in the following proposition.
For proofs and more details consult the survey \cite{miasnikov2007algebraic}
or Section 4 in \cite{PP15}.
\begin{prop}
\label{prop:properties of algebraic extensions}Let $\F$ be a finitely
generated free group.
\begin{enumerate}
\item Algebraic extensions form a partial order on the set of subgroups
of $\F$. In particular, $H\alg H$ for all $H$, and $H\alg K$ whenever
$H\alg J$ and $J\alg K$.
\item If $H\le J\le K$ and $H\alg K$ then $J\alg K$.
\item \label{enu:uniqe factorization to alg+free}For every extension of
free groups $H\le J$ there is a unique intermediate subgroup $A$
satisfying $H\alg A\ff J$. Moreover, every algebraic extension of
$H$ which is contained in $J$, is also contained in $A$.
\item Every finitely generated subgroup of $\F$ has finitely many algebraic
extensions.
\end{enumerate}
\end{prop}

\subsection{Fixed points of random permutations and Möbius inversions}

In \cite{PP15}, the main object of study is $\trw\left(N\right)$,
the expected number of fixed point in a $w$-random permutation in
$S_{N}$. We discuss here parts of the analysis in \cite{PP15} which
are relevant for the proof of Theorem \ref{thm:free words inside a word}.

First, as explained at the beginning of Section \ref{sec:profinite},
a $w$-random permutation can be obtained as $\varphi\left(w\right)$,
where $\varphi\in\mathrm{Hom}\left(\F,S_{N}\right)$ is a uniformly
random homomorphism. In a similar manner, for every subgroup $H\le\F$,
one can define an $H$-random subgroup of $S_{N}$ as $\varphi\left(H\right)$,
the image of $H$ through a random homomorphism. Denote by $\Phi_{H,\F}\left(N\right)$
the expected number of elements in $\left\{ 1,\ldots,N\right\} $
which are fixed by all permutations in $\varphi\left(H\right)$. In
particular, $\trw\left(N\right)=\Phi_{\left\langle w\right\rangle ,\F}\left(N\right)$.
This notion can be then defined for every pair of finitely generated
subgroups:
\begin{defn}
Let $H,J\le\F$ be finitely generated with $H\le J$. Denote by $\Phi_{H,J}\left(N\right)$\marginpar{$\Phi_{H,J}$}
the expected number of joint fixed points of all permutations in $\varphi\left(H\right)$
where $\varphi\in\mathrm{Hom}\left(J,S_{N}\right)$ is uniformly random.
\end{defn}

Clearly, if $A\ff\F$ is a free factor and $\varphi\in\mathrm{Hom}\left(\F,S_{n}\right)$
is uniformly random, then $\varphi\Big|_{A}\in\mathrm{Hom}\left(A,S_{n}\right)$
is also uniformly random. In this case, therefore,
\[
\Phi_{A,\F}\left(N\right)=\Phi_{A,A}\left(N\right)=N^{1-\mathrm{rank}\left(A\right)}.
\]
Likewise, if $H\alg A\ff\F$ is the unique factorization of the extension
$H\le\F$ to an algebraic extension and a free extension, then 
\begin{equation}
\Phi_{H,\F}\left(N\right)=\Phi_{H,A}\left(N\right).\label{eq:Phi can be computed on alg extensions only}
\end{equation}

Next, one can define a ``Möbius inversion'' of the function $\Phi$
based on the partial order ``$\alg$'' defined above. Assume that
$H\alg J$. Because every finitely generated subgroup $H\le\F$ has
only finitely many algebraic extensions, we have, in particular, that
there are only finitely many intermediate subgroups $M$ with $H\alg M\alg J$.
This allows us to define the ``right inversion'' (or derivation)
$R$ of the function $\Phi$, as follows:\marginpar{$R_{H,M}$}
\begin{equation}
\Phi_{H,J}\left(N\right)=\sum_{M\colon H\alg M\alg J}R_{H,M}\left(N\right).\label{eq:def of R}
\end{equation}
Indeed, this well-defines $R_{H,J}\left(N\right)$ by induction on
the number of intermediate subgroups in the poset defined by ``$\alg$'':
\begin{align}
R_{H,J}\left(N\right) & =\Phi_{H,J}\left(N\right)-\sum_{M\colon H\alg M\lneqq_{\mathrm{alg}}J}R_{H,M}\left(N\right).\label{eq:recursive def of R}
\end{align}

\begin{rem}
The initial definition of $R$ in \cite{PP15} is slightly different.
It is based on a different partial order ``$\covers$'' on the finitely
generated groups of $\F$, a partial order based on Stallings core
graphs and which is basis-dependent (here $X$ marks a given basis).
This order is ``finer'' then $\alg$, in the sense that $H\covers J$
whenever $H\alg J$. The resulting function is denoted there $R^{X}$.
However, it is then shown \cite[Proposition 5.1]{PP15} that $R^{X}$
is supported on algebraic extensions, that the value is independent
of the basis $X$, and that, in fact, it is equal to the function
defined in (\ref{eq:def of R}) \cite[Equation (5.3)]{PP15}.
\end{rem}

The main result of \cite{PP15} easily follows from the following
more technical statements about the function $R$.
\begin{thm}
\label{thm:guts of PP15}Assume that $H\alg J\le\F$ are all finitely
generated groups.
\begin{enumerate}
\item \cite[immediate corollary of Lemma 6.4]{PP15} For large enough $N$,
the function $R_{H,J}\left(N\right)$ is equal to a rational expression
in $N$.
\item \label{enu:R for algebraic extensions}\cite[Proposition 7.2]{PP15}
\[
R_{H,J}\left(N\right)=N^{1-\mathrm{rank}\left(J\right)}+O\left(N^{-\mathrm{rank}\left(J\right)}\right).
\]
\item \label{enu:R_H,H}By definition, 
\[
R_{H,H}\left(N\right)=\Phi_{H,H}\left(N\right)=N^{1-\mathrm{rank}\left(H\right)}.
\]
\end{enumerate}
\end{thm}

Let $H\le\F$ be finitely generated. Let $H\alg A\ff\F$ be the unique
factorization into an algebraic and a free extensions. Using items
\ref{enu:R for algebraic extensions} and \ref{enu:R_H,H} from Theorem
\ref{thm:guts of PP15}, we obtain that 
\begin{eqnarray}
\Phi_{H,\F}\left(N\right) & \stackrel{\eqref{eq:Phi can be computed on alg extensions only}}{=} & \Phi_{H,A}\left(N\right)\stackrel{\eqref{eq:def of R}}{=}\sum_{J\colon H\alg J\alg A}R_{H,J}\left(N\right)\nonumber \\
 & \stackrel{\mathrm{Proposition}~\ref{prop:properties of algebraic extensions}}{=} & \sum_{J\le\F\colon H\alg J}R_{H,J}\left(N\right)\nonumber \\
 & \stackrel{\mathrm{Theorem}~\ref{thm:guts of PP15}}{=} & N^{1-\mathrm{rank}\left(H\right)}+\sum_{J\le\F\colon H\lvertneqq_{\mathrm{alg}}J}\left[N^{1-\mathrm{rank}\left(J\right)}+O\left(N^{-\mathrm{rank}\left(J\right)}\right)\right].\label{eq:Phi in terms of alg extensions}
\end{eqnarray}
Equation (\ref{eq:Phi in terms of alg extensions}) leads immediately
to the following theorem, which is the main result of \cite{PP15}
with regards to $\trw\left(N\right)$.
\begin{thm}
\label{thm:PP15 main}\cite[Theorem 1.8]{PP15} Let $H\le\F$ be finitely
generated free groups. Denote by $\pi\left(H\right)$ the smallest
rank of a \emph{proper} algebraic extension of $H$, or $\pi\left(H\right)=\infty$
if there are no proper algebraic extensions, namely, if $H\ff\F$.
Then, 
\[
\Phi_{H,\F}\left(N\right)=N^{1-\mathrm{rank}\left(H\right)}+C\cdot N^{1-\pi\left(H\right)}+O\left(N^{-\pi\left(H\right)}\right),
\]
where $C$ is the number of proper algebraic extensions of $H$ of
rank $\pi\left(H\right)$.

In particular, for a word $w\in\F$, denote by $\pi\left(w\right)$
the smallest rank of a \emph{proper} algebraic extension of $\left\langle w\right\rangle $,
or $\pi\left(w\right)=\infty$ if $w$ is primitive in $\F$. Then,
\[
\trw\left(N\right)=\Phi_{\left\langle w\right\rangle ,\F}\left(N\right)=1+C\cdot N^{1-\pi\left(w\right)}+O\left(N^{-\pi\left(w\right)}\right),
\]
where $C$ is the number of proper algebraic extensions of $\left\langle w\right\rangle $
of rank $\pi\left(w\right)$.
\end{thm}

The following theorem is a generalization of Theorem \ref{thm:PP15 main},
which can also be seen as a quantitative version of Theorem \ref{thm:free words inside a word}.
\begin{thm}
\label{thm:generalization of pi}Let $H\alg J$ be an algebraic extension
of finitely generated free groups. Let $\iota\colon J\hookrightarrow\F=\F_{r}$
be an embedding of $J$ in $\F$. Denote by $\pi_{\iota}\left(H\right)$\marginpar{$\pi_{\iota}\left(H\right)$}
the smallest rank of an algebraic extension of $\iota\left(H\right)$
in $\F$ which is \emph{not contained in $\iota\left(J\right)$, or
$\pi_{\iota}\left(H\right)=\infty$ if $\iota\left(J\right)\ff\F$.
Then,} 
\[
\Phi_{\iota\left(H\right),\F}\left(N\right)=\begin{cases}
\Phi_{H,J}\left(N\right)+C\cdot N^{1-\pi_{\iota}\left(H\right)}+O\left(N^{-\pi_{\iota}\left(H\right)}\right) & \mathrm{if}\,\,\pi_{\iota}\left(H\right)<\infty\\
\Phi_{H,J}\left(N\right) & \mathrm{if}\,\,\pi_{\iota}\left(H\right)=\infty
\end{cases},
\]
where $C$ is the number of algebraic extensions of $\iota\left(H\right)$
of rank $\pi_{\iota}\left(H\right)$ inside $\F$ not contained in
$\iota\left(J\right)$. In particular, if $\pi_{\iota}\left(H\right)<\infty$
then for every large enough $N$, we have 
\[
\Phi_{\iota\left(H\right),\F}\left(N\right)>\Phi_{H,J}\left(N\right).
\]
\end{thm}

If $H$ is a subgroup of $\F$ and we let $J=H$ and $\iota\colon H\hookrightarrow\F$
be the embedding, then Theorem \ref{thm:generalization of pi} reduces
to Theorem \ref{thm:PP15 main}.
\begin{proof}[Proof of Theorem \ref{thm:generalization of pi}]
 Let $\iota\left(H\right)\alg A\ff\F$ be the unique factorization
of $\iota\left(H\right)\le\F$ to an algebraic extension and a free
extension. By (\ref{eq:Phi can be computed on alg extensions only}),(\ref{eq:def of R})
and Proposition \ref{prop:properties of algebraic extensions},
\begin{eqnarray}
\Phi_{\iota\left(H\right),\F}\left(N\right) & = & \Phi_{\iota\left(H\right),A}\left(N\right)=\sum_{M\colon\iota\left(H\right)\alg M\alg A}R_{\iota\left(H\right),M}\left(N\right)=\sum_{M\le\F\colon\iota\left(H\right)\alg M}R_{\iota\left(H\right),M}\left(N\right)\nonumber \\
 & = & \sum_{M\le\F\colon\iota\left(H\right)\alg M\le\iota\left(J\right)}R_{\iota\left(H\right),M}\left(N\right)+\sum_{M\le\F\colon\iota\left(H\right)\alg M\nleq\iota\left(J\right)}R_{\iota\left(H\right),M}\left(N\right).\label{eq:two summands}
\end{eqnarray}
By Theorem \ref{thm:guts of PP15}(\ref{enu:R for algebraic extensions}),
the second summand in (\ref{eq:two summands}) is precisely $C\cdot N^{1-\pi_{\iota}\left(H\right)}+O\left(N^{-\pi_{\iota}\left(H\right)}\right)$,
where $C$ is as in the statement of the theorem. It remains to show
that the first summand is equal to $\Phi_{H,J}\left(N\right)$.

Indeed, as $\iota$ is an isomorphism between $J$ and $\iota\left(J\right)$,
the algebraic extensions of $H$ in $J$ are precisely $\left\{ \iota^{-1}\left(M\right)\,\middle|\,\iota\left(H\right)\alg M\le\iota\left(J\right)\right\} $.
It is thus enough to show that for every $M$ with $\iota\left(H\right)\alg M\le\iota\left(J\right)$
(and every $N$) we have
\[
R_{H,\iota^{-1}\left(M\right)}\left(N\right)=R_{\iota\left(H\right),M}\left(N\right).
\]
But by the definition of the derivation $R$ in (\ref{eq:def of R}),
the values $\left\{ R_{H,K}\left(N\right)\right\} _{H\alg K\le J}$
are given by the values $\left\{ \Phi_{H,K}\left(N\right)\right\} _{H\alg K\le J}$,
and it is clear that if $H\alg K\le J$ then $\Phi_{H,K}\left(N\right)=\Phi_{\iota\left(H\right),\iota\left(K\right)}\left(N\right)$.
\end{proof}

\begin{proof}[Proof of Theorem \ref{thm:free words inside a word}]
 Recall the assumptions of Theorem \ref{thm:free words inside a word}:
$w\in\F_{k}$ is not contained in a proper free factor, and $u_{1},\ldots,u_{k}\in\F$
are free and do not generate a free factor of $\F$. The assumption
that $w$ is not contained in a proper free factor is equivalent to
that $\left\langle w\right\rangle \alg\F_{k}$. If $y_{1},\ldots,y_{k}$
is a basis of $\F_{k}$, define a map $\iota\colon\F_{k}\to\F$ by
$\iota\left(y_{i}\right)=u_{i}\in\F$ for $1\le i\le k$. The assumption
that $u_{1},\ldots,u_{k}$ are free is equivalent to that $\iota\colon\F_{k}\to\F$
is an embedding. Finally, $u_{1},\ldots,u_{k}$ generate a free factor
of $\F$ if and only if $\pi_{\iota}\left(\left\langle w\right\rangle \right)=\infty$.
So if $u_{1},\ldots,u_{k}$ do not generate a free factor then $\pi_{\iota}\left(\left\langle w\right\rangle \right)<\infty$
and, by Theorem \ref{thm:generalization of pi}, we obtain
\[
\tr_{w\left(u_{1},\ldots,u_{k}\right)}\left(N\right)=\Phi_{\iota\left(\left\langle w\right\rangle \right),\F}\left(N\right)=\Phi_{\left\langle w\right\rangle ,\F_{k}}\left(N\right)+C\cdot N^{1-\pi_{\iota}\left(\left\langle w\right\rangle \right)}+O\left(N^{-\pi_{\iota}\left(\left\langle w\right\rangle \right)}\right),
\]
which is strictly larger than $\Phi_{\left\langle w\right\rangle ,\F_{k}}\left(N\right)=\trw\left(N\right)$
for every large enough $N$.
\end{proof}
\begin{rem}
As explained towards the end of Section \pageref{sec:Introduction},
Theorem \ref{thm:free words inside a word} can be used to show that
if the word $\left[u,v\right]$ (for some $u,v\in\F$) induces the
same measures on finite groups as $\left[x,y\right]$, then $\left\{ u,v\right\} $
can be extended to a basis of $\F$, and that if $u^{m}$ induces
the same measures on finite groups as $x^{m}$, then $u$ is primitive.

Another case where Theorem \ref{thm:free words inside a word} is
handy is the case of surface words analyzed in \cite{MP-surface-words}.
Consider first a word $w$ which induces the same measures on \emph{compact
}groups as the orientable surface word $s_{g}=\left[x_{1},y_{1}\right]\cdots\left[x_{g},y_{g}\right]$.
The proof in \cite{MP-surface-words} that $w\aute{\F}s_{g}$ consists
of three steps. First, using measures on unitary groups, it is shown
that $w$ is a product of at most $g$ commutators. Then, using measures
on a generalized symmetric group $S^{1}\wr S_{N}$, it is shown that
$w$ is in fact a product of $g$ commutators, so $w=\left[u_{1},v_{1}\right]\cdots\left[u_{g},v_{g}\right]$
such that $u_{1},v_{1},\ldots,u_{g},v_{g}$ are free (see the section
``Overview of the proof'' in \cite[Page 5]{MP-surface-words}).
Then, the more involved \cite[Theorem 3.6]{MP-surface-words} is used
to finish the proof. However, for this last step, one can also use
Theorem \ref{thm:free words inside a word}, from which it follows
that if $w=\left[u_{1},v_{1}\right]\cdots\left[u_{g},v_{g}\right]$
with $u_{1},v_{1},\ldots,u_{g},v_{g}$ free and $w$ induces the same
measures on $S_{N}$ as $s_{g}$, then $w\aute{\F}s_{g}$.

Theorem \ref{thm:free words inside a word} can be similarly used
for the other type of words studied in \cite{MP-surface-words}: that
of non-orientable surface words $x_{1}^{2}\cdots x_{g}^{2}$.
\end{rem}

\section{Measures induced by powers\label{sec:powers}}

\subsection{Proof of Theorem \ref{thm:from word to its powers}}

In the language of Profinite topology, Khelif's Theorem \ref{thm:Khelif}
says that the set of commutators in $\F$ is closed in the profinite
topology. There is a similar result, due to Lubotzky, concerning the
set of $d$th powers in $\F$. As explained in the paragraph following
Theorem \ref{thm:Khelif}, this immediately implies that $d$th powers
are distinguishable by measures from non-$d$th powers:
\begin{thm}[{Lubotzky, see \cite[Page 252]{thompson1997power}}]
\label{thm:nth powers are closed} The set $\left\{ u^{d}\,\middle|\,u\in\F\right\} $
of $d$th powers is closed in the profinite topology on $\F$. In
particular, if $w_{1}$ is a $d$th power and $w_{2}$ is not, then
$w_{1}$ and $w_{2}$ do not induce the same measure on all finite
groups.
\end{thm}

We include below (Theorem \ref{thm:Lubotzky with homomorphisms to S_N})
another proof of Theorem \ref{thm:nth powers are closed} which relies
on homomorphisms from the free group to $S_{N}$.

Another ingredient in our proof of Theorem \ref{thm:from word to its powers}
is the following theorem due to Herfort and Ribes. The free profinite
product of two profinite groups $A$ and $B$ is denoted $A\sqcup B$
-- see \cite[Section 9.1]{ribes2010profinite} for a discussion on
free profinite products. We denote the centralizer of $g$ in the
group $G$ by $C_{G}\left(g\right)$.
\begin{thm}
\label{thm:centralizers in free products}\cite[Theorem B]{herfort1985torsion}
Let $A$ and $B$ be profinite groups and let $A\sqcup B$ be their
free profinite product. If $a\in A$ then the centralizer of $a$
in $A\sqcup B$ is contained in $A$.
\end{thm}

For a subset $S$ of $\F$ we denote by $\overline{S}^{\hat{\F}}$
the closure of $\iota\left(S\right)$ in the profinite completion
$\hat{\F}$.
\begin{lem}
\label{lem:centralizer of an element in hat F}For every $1\neq w\in\F$,
$C_{\hat{\F}}\left(w\right)=\overline{C_{\F}\left(w\right)}^{\hat{\F}}$.
In particular, $C_{\hat{\F}}\left(w\right)\cong\hat{\mathbb{Z}}$.
\end{lem}

\begin{proof}
Let $u\in\F$ be a root of $w$ which is a non-power. So $C_{\F}\left(w\right)=\left\langle u\right\rangle \cong\mathbb{Z}$.
By Hall's theorem (see \cite[Proposition I.3.10]{Lyndon1977}), $u$
can be extended to a basis $\left\{ u=u_{1},u_{2},\ldots,u_{t}\right\} $
of a finite index subgroup $H\le\F$. If $A$ and $B$ are abstract
groups, then $\hat{A}\sqcup\hat{B}=\widehat{A*B}$ \cite[Exercise 9.1.1]{ribes2010profinite}.
So if we denote by $\overline{H}=\overline{H}^{\hat{\F}}$ the closure
of $H$ in $\hat{\F}$, then $\overline{H}=\overline{\left\langle u\right\rangle }^{\hat{\F}}\sqcup\overline{\left\langle u_{2},\ldots,u_{t}\right\rangle }^{\hat{\F}}$,
and by Theorem \ref{thm:centralizers in free products}, $C_{\overline{H}}\left(w\right)\le\overline{\left\langle u\right\rangle }^{\hat{\F}}\cong\hat{\mathbb{Z}}$.
But $\hat{\mathbb{Z}}$ is abelian, so $C_{\overline{H}}\left(w\right)=\overline{\left\langle u\right\rangle }^{\hat{\F}}=\overline{C_{\F}\left(w\right)}^{\hat{\F}}\cong\hat{\mathbb{Z}}$.

It remains to show that for every $b\in\hat{\F}\setminus\overline{H}$,
$b$ does not commute with $w$. Pick $e=b_{1},b_{2},\ldots,b_{s}\in\F$
representatives for the left cosets of $H$ in $\F$, which are thus
also representatives for the left cosets of $\overline{H}$ in $\hat{\F}$.
Assume that $b\in b_{i}\overline{H}$ for some $2\le i\le s$. If
$b_{i}^{-1}wb_{i}\notin\overline{H}$, then $b^{-1}wb\notin\overline{H}$,
so we may assume that $b_{i}^{-1}wb_{i}\in\overline{H}$ and thus
also $b^{-1}wb\in\overline{H}$. As $C_{\F}\left(w\right)=\left\langle u\right\rangle \le H$
and $b_{i}\notin H$, $b_{i}^{-1}wb_{i}$ is not conjugate to $w$
in $H$. Because free groups are conjugacy-separable (e.g.~\cite[Proposition I.4.8]{Lyndon1977}),
$b_{i}^{-1}wb_{i}$ is not conjugate to $w$ also in $\overline{H}$.
So $b^{-1}wb\ne w$.
\end{proof}
\begin{cor}
\label{cor:every profinite root of a word is in F}Every root of $w\in\F$
in $\hat{\F}$ belongs to $\F$.
\end{cor}

\begin{proof}
Since $\hat{\F}$ is torsion free, we may assume $w\ne1$. Assume
that $w=x^{m}$ with $x\in\hat{\F}$ and $m\in\mathbb{Z}_{\ge1}$.
By Theorem \ref{thm:nth powers are closed}, the set of $m$th powers
in $\F$ is closed in the profinite topology, which means that $w=v^{m}$
for some $v\in\F$. Let $u\in\F$ be a non-power such that $v=u^{\ell}$
for some $\ell\in\mathbb{Z}_{\ge1}$. By Lemma \ref{lem:centralizer of an element in hat F},
$C_{\hat{\F}}\left(w\right)=\overline{\left\langle u\right\rangle }^{\hat{\F}}\cong\hat{\mathbb{Z}}$.
As $x,v\in C_{\hat{\F}}\left(w\right)$ and $\hat{\mathbb{Z}}$ is
abelian, $x$ and $v$ commute. So $\left(x^{-1}v\right)^{m}=x^{-m}v^{m}=1$.
But $\hat{\mathbb{Z}}$ is torsion-free and thus $x=v$.
\end{proof}
We now have the tools to prove Theorem \ref{thm:from word to its powers}.
\begin{proof}[Proof of Theorem \ref{thm:from word to its powers}]
 The theorem is clear for $d=0$. Now let $d\in\mathbb{Z}\setminus\left\{ 0\right\} $
and assume that $w\in\F$ is profinitely rigid and that $w_{2}\in\F$
induces the same measures on finite groups as $w^{d}$. From Theorem
\ref{thm:nth powers are closed} it follows that $w_{2}$ is a $d$th
power, so $w_{2}=v^{d}$ for some $v\in\F$. By Theorem \ref{thm:six equivalences}
there is an automorphism $\theta\in\mathrm{Aut}\hat{\F}$ with $\theta\left(w^{d}\right)=v^{d}$.
So $\theta\left(w\right)$ is a $d$th root of $v^{d}$ in $\hat{\F}$,
and by Corollary \ref{cor:every profinite root of a word is in F},
$\theta\left(w\right)=v$, namely, $w\aute{\hat{\F}}v$. But $w$
is profinitely rigid so $w\auteq v$. Thus also $w^{d}\aute{\F}v^{d}=w_{2}$.
\end{proof}

\subsection{Powers in symmetric groups}

For completeness, we also include a proof of the fact that it is enough
to consider measures on the symmetric groups $S_{N}$ (for all $N$)
to obtain the case of primitive powers in Theorem \ref{thm:main}.
Note that this is not true for the words $\left[x,y\right]^{d}$.
Indeed, the word $\left[x,y\right]=xyx^{-1}y^{-1}$ and the word $xyxy^{-1}$,
which lie in different $\mathrm{Aut}\F$-orbits, induce the exact
same measure on $S_{N}$ for all $N$: in both words one takes a uniformly
random permutation (the image of $x$) and multiplies it with a uniform
random conjugate (the image of $yx^{-1}y^{-1}$ or of $yxy^{-1}$).
It is plausible that measures on the alternating groups $\mathrm{Alt}\left(N\right)$
do distinguish the orbit of $\left[x,y\right]$ from all other orbits,
but we do not know whether this is true or not.
\begin{prop}
\label{prop:primitive powers from S_N}If $w$ induces the same measures
as $x^{d}$ on the symmetric group $S_{N}$ for all $N$, then $w\auteq x^{d}$.
\end{prop}

We begin with a lemma which identifies powers in $S_{N}$. Denote
by $c_{t}\left(\sigma\right)$ the number of $t$-cycles in the cycle
decomposition of $\sigma\in S_{N}$, and for a prime $p$ and a positive
integer $n$, denote
\[
\nu_{p}\left(n\right)=\max\left\{ e\in\mathbb{Z}_{\ge0}~\,\middle|\,~~p^{e}\mid n\right\} .
\]

\begin{lem}
\label{lem:when a permutation is a dth power}A permutation $\sigma\in S_{N}$
is a $d$th power if and only if for all $t\in\mathbb{Z}_{\ge1}$
we have
\[
\left(\prod_{p\mid t}p^{\nu_{p}\left(d\right)}\right)\mid c_{t}\left(\sigma\right),
\]
the product being over all prime divisors of $t$. In particular,
when $d\mid t$, the condition on $c_{t}$ is that $d\mid c_{t}\left(\sigma\right)$.
\end{lem}

\begin{proof}
First, assume that $\sigma$ is a $d$th power. The $d$th power of
an $\ell$-cycle is the union of $\gcd\left(d,\ell\right)$ cycles,
each of length $\frac{\ell}{\gcd\left(d,\ell\right)}$. We need to
show that if $p\mid t$ and $t=\frac{\ell}{\gcd\left(d,\ell\right)}$
then $p^{\nu_{p}\left(d\right)}\mid\gcd\left(d,\ell\right)$. But
if $p\mid\frac{\ell}{\gcd\left(d,\ell\right)}$ then $\nu_{p}\left(\ell\right)>\nu_{p}\left(d\right)$
and so $\nu_{p}\left(\gcd\left(d,\ell\right)\right)=\nu_{p}\left(d\right)$.

For the converse implication, it is enough to prove the claim in the
case where $\sigma$ is simply the product of $\left(\prod_{p\mid t}p^{\nu_{p}\left(d\right)}\right)$
disjoint cycles of length $t$. In this case, there is a cycle of
length $t\cdot\left(\prod_{p\mid t}p^{\nu_{p}\left(d\right)}\right)$
whose $d$th power is $\sigma$.
\end{proof}
\begin{lem}
\label{lem:number of t-cycles in an a-power}Assume that $b,t\in\mathbb{Z}_{\ge1}$
with $b\mid t$. Let $N\ge2bt$ and let $\sigma\in S_{N}$ be a uniform
random permutation. Then
\[
\mathbb{E}\left[c_{t}\left(\sigma^{b}\right)\right]=\frac{1}{t},~~\mathrm{and}~~\mathbb{E}\left[c_{t}^{~2}\left(\sigma^{b}\right)\right]=\frac{b}{t}+\frac{1}{t^{2}}.
\]
\end{lem}

\begin{proof}
As $b\mid t$, a $t$-cycle in $\sigma^{b}$ must come from a $bt$-cycle
in $\sigma$, and each $bt$-cycle in $\sigma$ gives rise to $b$
disjoint cycles of length $t$ in $\sigma^{b}$. Hence $c_{t}\left(\sigma^{b}\right)=bc_{bt}\left(\sigma\right)$.
When $N\ge bt$, the expected number of $bt$-cycles in $\sigma$
in $\frac{1}{bt}$, and so $\mathbb{E}\left[c_{t}\left(\sigma^{b}\right)\right]=b\mathbb{E}\left[c_{bt}\left(\sigma\right)\right]=\frac{1}{t}$.
Likewise, when $N\ge2bt$, $\mathbb{E}\left[c_{bt}^{~2}\left(\sigma\right)\right]=\frac{1}{bt}+\frac{1}{b^{2}t^{2}}$,
so
\[
\mathbb{E}\left[c_{t}^{~2}\left(\sigma^{b}\right)\right]=b^{2}\mathbb{E}\left[c_{bt}^{~2}\left(\sigma\right)\right]=b^{2}\left(\frac{1}{bt}+\frac{1}{b^{2}t^{2}}\right)=\frac{b}{t}+\frac{1}{t^{2}}.
\]
\end{proof}
We can now give a proof of Lubotzky's Theorem \ref{thm:nth powers are closed}
using homomorphic images in symmetric groups.
\begin{thm}
\label{thm:Lubotzky with homomorphisms to S_N}Let $w\in\F$. If $\varphi\left(w\right)\in S_{N}$
is a $d$th power for every $N\in\mathbb{Z}_{\ge1}$ and every $\varphi\in\mathrm{Hom}\left(\F,S_{N}\right)$,
then $w$ is a $d$th power.
\end{thm}

\begin{proof}
The statement is trivial for $d=1$ or $w=1$, so assume $d\ge2$
and $w\ne1$. Let $w=u^{b}$ with $u\ne1$ a non-power. Denote $t=\mathrm{lcm}\left(b,d\right)$.
Every image $\varphi\left(w\right)$ of $w$ in $S_{N}$ is a $b$th
power, and by Lemma \ref{lem:when a permutation is a dth power},
as $b\mid t$,
\[
b\mid c_{t}\left(\varphi\left(w\right)\right).
\]
By assumption, every image $\varphi\left(w\right)$ of $w$ in $S_{N}$
is also a $d$th power, and, as $d\mid t$,
\[
d\mid c_{t}\left(\varphi\left(w\right)\right).
\]
We deduce that $t=\mathrm{lcm}\left(b,d\right)\mid c_{t}\left(\varphi\left(w\right)\right)$.
Thus 
\[
c_{t}^{~2}\left(\varphi\left(w\right)\right)\ge t\cdot c_{t}\left(\varphi\left(w\right)\right).
\]
Taking expectations and then taking the limit as $N\to\infty$, we
see that 
\begin{equation}
\lim_{N\to\infty}\mathbb{E}\left[c_{t}^{~2}\left(\varphi\left(w\right)\right)\right]\ge t\cdot\lim_{N\to\infty}\mathbb{E}\left[c_{t}^{~}\left(\varphi\left(w\right)\right)\right].\label{eq:limits for w=00003Du to the a}
\end{equation}
Recall that $w=u^{b}$ with $u$ a non-power. It is a theorem of Nica
\cite[Theorem 1.1]{nica1994number} that the random variables $c_{t}\left(\varphi\left(w\right)\right)$
(where $\varphi\in\mathrm{Hom}\left(\F,S_{N}\right)$ is uniformly
random) have a limit distribution which depends only on $b$ and not
on\footnote{It is more generally true that $c_{1}\left(\varphi\left(w\right)\right),c_{2}\left(\varphi\left(w\right)\right),\ldots,c_{\ell}\left(\varphi\left(w\right)\right)$
have a limit joint distribution which only depends on $a$ -- see
\cite[Remark 31]{Linial2010}.} $u$. In particular, the limits in (\ref{eq:limits for w=00003Du to the a})
remain unchanged when $w$ is replaced with $x^{b}$. By Lemma \ref{lem:number of t-cycles in an a-power}
this gives
\[
\frac{b}{t}+\frac{1}{t^{2}}\ge t\cdot\frac{1}{t}=1.
\]
Since $t=\mathrm{lcm}\left(b,d\right)$ and $d\ge2$ we must have
$b=t$, so $d\mid b$.
\end{proof}

\begin{proof}[Proof of Proposition \ref{prop:primitive powers from S_N}]
Assume that for all $N$, $w$ induces the same measures on $S_{N}$
as $x^{d}$. In particular, every image of $w$ through an homomorphism
to $S_{N}$ is a $d$th power, so by Theorem \ref{thm:Lubotzky with homomorphisms to S_N},
$w=v^{d}$ for some $v\in\F$. We are now done by Theorem \ref{thm:free words inside a word}
applied with the word $x^{d}$.
\end{proof}
\begin{rem}
The phenomenon observed by Nica that if $u$ is a non-power then
(moments of) $u^{d}$-measures converge to the same limits as $x^{d}$-measures
in $S_{N}$, is true in many families of groups. It is true in the
wreath products $C_{m}\wr S_{N}$ (as illustrated in \cite{MP-surface-words})
and for general linear groups over finite fields \cite{West}. It
is also true for families of infinite compact groups, such as unitary
groups (\cite{MSS07} or \cite[Corollary 1.13]{MP-Un}) and Orthogonal
and compact Symplectic groups \cite[Corollary 1.17]{MP-On}.
\end{rem}

\begin{rem}
Assume $1\ne u\in\F$ is a non-power. In the notation of Section \ref{sec:Fixed-points-of},
$\tr_{u^{d}}\left(N\right)=\Phi_{\left\langle u^{d}\right\rangle ,\F}\left(N\right)$,
and to use Theorem \ref{thm:generalization of pi} for this case,
we apply it with $H=\left\langle x^{d}\right\rangle $ and $J=\left\langle x\right\rangle $.
Then, 
\[
\tr_{u^{d}}\left(N\right)=\Phi_{\left\langle u^{d}\right\rangle ,\F}\left(N\right)=\Phi_{\left\langle x^{d}\right\rangle ,\left\langle x\right\rangle }\left(N\right)+C\cdot N^{1-\pi_{\iota}\left(\left\langle x^{d}\right\rangle \right)}+O\left(N^{-\pi_{\iota}\left(\left\langle x^{d}\right\rangle \right)}\right),
\]
with $\iota\colon J\to\F$ defined by $x\mapsto u$, $\pi_{\iota}\left(\left\langle x^{d}\right\rangle \right)$
the smallest rank of an algebraic extension of $\left\langle u^{d}\right\rangle $
not contained in $\left\langle u\right\rangle $, and $C$ the number
of such algebraic extensions of rank $\pi_{\iota}\left(\left\langle x^{d}\right\rangle \right)$.
In \cite{Hanany2019measures-on-Sn} the first and last author study
more generalizations of the results in \cite{PP15}, and, in particular,
show that in this case $\pi_{\iota}\left(\left\langle x^{d}\right\rangle \right)=\pi\left(u\right)$,
and all algebraic extensions of $\left\langle u^{d}\right\rangle $
of rank $\pi\left(u\right)$ are also algebraic extensions of $\left\langle u\right\rangle $.
If we denote $f_{u}\left(N\right)=\tr_{u^{d}}\left(N\right)-\tr_{u}\left(N\right)$,
we get that
\[
f_{u}\left(N\right)=f_{x}\left(N\right)+O\left(N^{-\pi\left(u\right)}\right)=\delta\left(d\right)-1+O\left(N^{-\pi\left(u\right)}\right),
\]
where $\delta\left(d\right)$ is the number of positive divisors of
$d$.
\end{rem}

\section*{Acknowledgments}

We thank Nir Avni, Lior Bary-Soroker, Arno Fehm and Pavel Zalesskii
for beneficial comments.

\bibliographystyle{alpha}
\bibliography{commutator_and_powers}

\noindent Liam Hanany, School of Mathematical Sciences, Tel Aviv University,
Tel Aviv, 6997801, Israel\\
\texttt{liamhanany@mail.tau.ac.il }~\\

\noindent Chen Meiri, Department of Mathematics, Technion - Israel
Institute of Technology, Haifa 3200003 Israel\\
chenm@technion.ac.il\\

\noindent Doron Puder, School of Mathematical Sciences, Tel Aviv University,
Tel Aviv, 6997801, Israel\\
\texttt{doronpuder@gmail.com}
\end{document}